\documentclass[10pt,a4paper]{article}
\usepackage[english]{babel}
\usepackage{algorithm}
\usepackage{algorithmic}
\usepackage{amsfonts}
\usepackage{amsmath}
\usepackage{theorem}
\usepackage{amssymb}
\usepackage{epsfig}
\usepackage{enumerate}
\usepackage{subfigure}
\usepackage{esint}

\newcommand{\CC}{\mathbb{C}}
\newcommand{\NN}{\mathbb{N}}
\newcommand{\RR}{\mathbb{R}}

\newcommand{\ints}{\int\limits}
\newcommand{\lam}{\lambda}

\newcommand{\vp}{\varphi}

\newcommand{\pa}{\partial}

\newcommand{\DD}{\mathcal{D}}

\DeclareMathOperator{\IM}{\mathrm{Im}}

\theoremstyle{plain}
\newtheorem{theorem}{Theorem}[section]

\newtheorem{lemma}[theorem]{Lemma}
\newtheorem{proposition}[theorem]{Proposition}

\theoremstyle{definition}

\newenvironment{proof}{
  \noindent{\it Proof.}\ }{\hspace*{\fill}
  \begin{math}\Box\end{math}\medskip}

\newcommand{\tablefont}[1]{\fontsize{2.5mm}{2.5mm}\selectfont}
\usepackage{array}
\newcolumntype{C}[1]{>{\centering\let\newline\\\arraybackslash\hspace{0pt}}m{#1}}
\newcolumntype{L}[1]{>{\raggedright\let\newline\\\arraybackslash\hspace{0pt}}m{#1}}
\newcolumntype{R}[1]{>{\raggedleft\let\newline\\\arraybackslash\hspace{0pt}}m{#1}}

\begin{document}

\title{A computational method for the Helmholtz equation in unbounded domains based on the minimization of an integral functional}

\author{Giulio Ciraolo\thanks{\footnotesize Department of Mathematics and Informatics, Universit\`a di Palermo, Via Archirafi 34, 90123 Palermo, Italy. \newline E-email: g.ciraolo@math.unipa.it, gargano@math.unipa.it, sciacca@math.unipa.it} \thanks{\footnotesize Corresponding author. \emph{E-mail address}: g.ciraolo@math.unipa.it (G. Ciraolo).}  \and Francesco Gargano\footnotemark[1] \and Vincenzo Sciacca\footnotemark[1]}

\maketitle

\begin{abstract}
We study a new approach to the problem of transparent boundary
conditions for the Helmholtz equation in unbounded domains. Our
approach is based on the minimization of an integral functional arising from a volume integral formulation of
the radiation condition. The index of refraction does not need to be constant at infinity and may have some angular dependency as well as perturbations. We prove analytical results on the convergence of the approximate solution. Numerical examples for different shapes of the artificial boundary and for non-constant indexes of refraction will be presented.
\end{abstract}

\section{Introduction}
Many problems in physical applications are modeled by wave propagation in unbounded domains. Computing the numerical solution is a challenging problem and, usually, one has to introduce an artificial boundary to make the computational domain finite. The problem of determining the boundary conditions on the artificial boundary has attracted many scientists and many methods have been studied leading to a very exciting field of mathematical research.

In this paper, we consider time-harmonic electromagnetic wave propagation in unbounded domains. Let $D\subset \RR^d,\: d \geq 2,$ be a (possibly empty) bounded domain. We consider the {\it Helmholtz equation}
\begin{equation}\label{eq: introd helm 1}
\Delta u + k^2n(x)^2 u = f,\quad \textmd{in } \RR^d \setminus \overline{D},
\end{equation}
where $k>0$ is called the {\it wavenumber}, $n>0$ is the index of refraction (which satisfies some assumptions to be specified later), $f$ is a source term and $u$ is part of the electromagnetic field. If $D$ is not empty, boundary conditions on $\partial D$ are given. To ensure the uniqueness of \eqref{eq: introd helm 1}, it is necessary to introduce a radiation condition at infinity. If $n\equiv 1$ in the exterior of some ball, we can consider the following Sommerfeld radiation condition
\begin{equation}\label{eq: introd rad cond Somm}
\lim_{r\to \infty} r^{\frac{d-1}{2}}\Big(\frac{\pa u}{\pa r} - iku \Big) =0,
\end{equation}
where $r=|x|$ is the distance from the origin and $i$ is the imaginary unit. The radiation condition \eqref{eq: introd rad cond Somm} follows from the expression of the far-field of the radiating solution. We will often refer as the solution of \eqref{eq: introd helm 1}-\eqref{eq: introd rad cond Somm} as the \emph{outgoing solution}. We shall discuss later the case when the index of refraction $n$ is not constant outside any bounded domain.

Depending on the assumptions on $n$, an exact formula for the solution of \eqref{eq: introd helm 1} may or may not exist (see for instance \cite{CK},\cite{Le}); however, when it exists, it may be intractable from a computational point of view. Thus, it is usually convenient to introduce an artificial computational \emph{bounded} domain $\Omega$, with $\Gamma= \partial \Omega$, and then impose some boundary conditions on $\Gamma$ in such a way that the solution $u_\Omega$ of the new problem approximates $u$ in a good fashion, i.e. one has to try to eliminate spurious reflections of waves arising from the introduction of the artificial boundary $\Gamma$.

As it is well-known, a large amount of work has been done on this kind of problems. The most used methods are based on \emph{local or nonlocal conditions} involving $u$ (\cite{BT},\cite{EM},\cite{Gi4}), approximations of the \emph{Dirichlet-to-Neumann map} (\cite{KG},\cite{GiK},\cite{GrK}), \emph{infinite element schemes} (\cite{Be1},\cite{Be2}), \emph{boundary element methods} (see \cite{SS}) and the \emph{perfectly matched layer method} (\cite{Ber},\cite{TY}). Many papers have been written of this subjects and for a deeper understanding of these problems and more recent developments, the interested reader can refer to \cite{As},\cite{Be2},\cite{Gi1},\cite{Gi2},\cite{Gi3},\cite{Hag},\cite{Har},\cite{Ih},\cite{Ts} and references therein.

In this paper, we propose a new approach to the study of wave propagation problems in unbounded domains. Let us start from the simplest case of finding the \emph{outgoing solution} $u$ of the Helmholtz equation in $\RR^d$ with constant index of refraction, i.e. to find $u$ satisfying
\begin{equation} \label{eq: introd Helm toy}
\Delta u + k^2 u = f,\quad \textmd{in } \RR^d,
\end{equation}
and the radiation condition \eqref{eq: introd rad cond Somm}.

It is well known that other radiation conditions than \eqref{eq: introd rad cond Somm} can be imposed at infinity to have the well-posedness of the problem. One of the most used is the following Rellich-Sommerfeld radiation condition
\begin{equation*}
\lim_{r\to \infty} \ints_{\pa B_r} \Big| \frac{\pa u}{\pa r} - iku \Big|^2 d\sigma = 0,
\end{equation*}
which was proposed by Rellich in \cite{Rel}. In the same paper, Rellich proved that a radiation condition can
also be given in the form
\begin{equation}\label{eq: introd rellich}
\ints_{\RR^d} \Big| \frac{\pa u}{\pa r} - ik u \Big|^2 dx <
+\infty.
\end{equation}
Condition \eqref{eq: introd rellich} can be considered as the starting point for our work, as we are going to explain shortly.

A naive way to look at problem \eqref{eq: introd Helm toy}-\eqref{eq: introd rellich} is the following:
\begin{equation}\label{eq: introd Problem Toy}
\textmd{ Minimize } \ints_{\RR^d} \Big| \frac{\pa u}{\pa r} - ik u \Big|^2 dx \textmd{ among the solutions of } \Delta u + k^2 u = f \textmd{ in } \RR^d.
\end{equation}
Clearly, the outgoing solution is the unique minimizer of \eqref{eq: introd Problem Toy}: it is the only solution of \eqref{eq: introd Helm toy} that satisfies \eqref{eq: introd rellich}, since the integral in \eqref{eq: introd rellich} does not converge for any other solution of \eqref{eq: introd Helm toy}. \\
Hence, when we introduce an artificial boundary $\Omega$, it is natural to consider the following constrained optimization problem:
\begin{equation}\label{eq: introd Problem Toy bounded}
\inf \Big\{\ints_{\Omega} \Big| \frac{\pa v}{\pa r} - ik v \Big|^2 dx:\ \  \Delta v + k^2 v = f \ \textmd{ in } \Omega \Big\}.
\end{equation}
As the domain $\Omega$ enlarges, it is reasonable to think that the minimizer of \eqref{eq: introd Problem Toy bounded} approximates the solution of \eqref{eq: introd Helm toy}-\eqref{eq: introd rellich}. However, for a constant index of refraction, available methods in literature are very efficient and well-established. For this reason, in this paper we shall consider more general indexes of refraction.

We now explain in details our method and the outline of the paper. We consider the Helmholtz equation
\begin{equation}\label{eq: introd Helm Perth Vega}
\Delta u + k^2 n(x)^2u = f, \quad \textmd{in } \RR^d;
\end{equation}
here the index of refraction $n$ is not necessarily constant at infinity but can have
an angular dependency like
\begin{equation*}
n(x) \to n_\infty(x/|x|) \quad \textmd{as } |x| \to +\infty.
\end{equation*}
Under some appropriate assumptions on this convergence and on $n_\infty$, in \cite{PV} (see Theorem 1.4) the authors proved that a Sommerfeld condition at infinity still holds true under the form
\begin{equation}\label{eq: intro rad con Perth Vega II}
\ints_{\RR^d} \Big| \nabla u(x) - ikn(x) u(x) \frac{x}{|x|} \Big|^2 \frac{dx}{1+|x|} < + \infty.
\end{equation}
There are other conditions of this type available in literature; we shall not discuss about them here and refer to \cite{PV},\cite{Sa},\cite{Zu} and references therein.

Let $\Omega$ be a bounded domain and consider the following minimization problem
\begin{equation}\label{eq: introd minimiz problem PV}
\inf \{ J_\Omega[v]:\ v \textmd{ satifies } \Delta v + k^2 n(x)^2 v = f \textmd{ in } \Omega\},
\end{equation}
where
\begin{equation}\label{eq: introd Funzionale}
J_\Omega[v]= \ints_{\Omega} \Big| \nabla v(x) - ikn(x) v(x) \frac{x}{|x|} \Big|^2 \frac{dx}{1+|x|}.
\end{equation}
The minimizer $u_\Omega$ of \eqref{eq: introd minimiz problem PV} is the approximation for the solution $u$ of \eqref{eq: introd Helm Perth Vega}-\eqref{eq: intro rad con Perth Vega II} that we propose in this work.

Problem \eqref{eq: introd minimiz problem PV} can be easily implemented numerically: we have to minimize a quadratic functional in a bounded domain subject to a linear constrain. We shall prove that when the computational domain enlarges and tends to cover the whole space, the approximating solution $u_\Omega$ converges in $H_{loc}^1$ norm to the outgoing solution. However, when $\Omega$ is a ball and $n$ is constant, most of the numerical methods cited before are more suitable and efficient for computing the solution than our approach.

When the artificial domain $\Omega$ can not be chosen as a ball or as other domains with simple geometry, most of the methods available in literature may present some difficulty in their applications. Since we do not use any technique of separation of variables or analytic continuation, the implementation of our approach do not suffer the change of geometry of the artificial domain and we will show by numerical examples that such a change does not affect the computations.

Our method can be applied also when the index of refraction is not constant outside any compact domain. In particular, we are thinking to an index of refraction that has some angularly dependency and may have some perturbation extending towards infinity. This is another relevant case in which our method is still applicable and of easy implementation. On the contrary, many of the methods cited before can not be applied (at least in a standard way).

The paper is organized as follows. In Section \ref{section 2} we prove a result of convergence for $u_\Omega$ and other analytical results when the artificial domain is a ball. In Section \ref{section 3} we give some numerical results on the convergence of the solution for a constant index of refraction for several choices of the artificial domain. In Section \ref{section 4} we present some numerical examples for non-constant indexes of refraction.

\section{Mathematical framework and analytical results} \label{section 2}
In this section we assume that the artificial domain is a ball and prove results of convergence as the radius of the ball goes to infinity under quite general assumptions on $n$. To simplify the exposition, firstly we study equation \eqref{eq: introd helm 1} with $D=\emptyset$ and then the convergence result is generalized to the case $D \neq \emptyset $ with Dirichlet or Neumann boundary conditions on $\partial D$.

Let $U:\RR^d \to \CC$ be the solution of problem \eqref{eq: introd Helm Perth Vega}-\eqref{eq: intro rad con Perth Vega II}. In what follows, we shall assume that \begin{equation}\label{hp n 0}
n \in L^\infty(\RR^d) \quad \textmd{and} \quad n(x) \geq n_* >0 \textmd{ for every } x \in \RR^d,
\end{equation}
for some positive $n_*$. We mention that the existence and uniqueness of the solution $U$ is a challenging problem and they are usually obtained by using a limiting absorption principle. In this paper we will assume that $U$ exists and is unique and refer to the papers cited in the Introduction for the appropriate assumptions on $n$.

Let $B_R \subset \RR^d$ be the ball of radius $R$ centered at the origin. For a solution $u$ of
\begin{equation}\label{eq: section 2 Helm Omega}
\Delta u + k^2 n(x)^2 u = f \quad \textmd{in } B_R,
\end{equation}
we mean a function $u\in H^1(B_R)$ satisfying the variational formulation of \eqref{eq: section 2 Helm Omega}:
\begin{equation}\label{eq: section 2 Helm var formul}
-\ints_{B_R} \nabla u \cdot \nabla \overline{\vp} dx + k^2 \ints_{B_R} n(x)^2 u \overline{\vp}dx =
\ints_{B_R} f \overline{\vp}dx, \quad  \forall\ \vp \in C_0^1(B_R).
\end{equation}
We denote by $\mathcal{D}_R$ the set of solutions of \eqref{eq: section 2 Helm var formul} in $B_R$:
\begin{equation}\label{eq: section 2 D_R}
\mathcal{D}_R= \{ u \in H^1(B_R):\ u\ \textmd{ satifies } \eqref{eq: section 2 Helm var formul} \textmd{ in } B_R \}.
\end{equation}
Clearly, $\mathcal{D}_R$ is a convex set. For $u\in H^1(B_R)$ we define the following seminorm:
\begin{equation}\label{eq: section 2 seminorm}
[u]_R = \left( \ \ints_{B_R} \Big| \nabla u - ikn u \frac{x}{|x|} \Big|^2 \frac{dx}{1+|x|} \right)^{\frac{1}{2}}.
\end{equation}
We consider the following family of minimization problems
\begin{equation}\label{eq: section 2 minimiz pb}
\inf \{[u]_R:\ u \in \mathcal{D}_R \},
\end{equation}
set
\begin{equation}\label{eq: section 2 d_R}
d_R=\inf \{[u]_R:\ u \in \mathcal{D}_R \}
\end{equation}
and denote by $u_{B_R}$ the minimizer of \eqref{eq: section 2 minimiz pb} (we will prove that $u_{B_R}$ exists and is unique).

\begin{lemma}\label{lemma imaginary part}
Let $u$ be a solution of
\begin{equation}\label{eq: section 2 helm omogenea}
\Delta u + k^2n(x)^2u=0\quad \textmd{in } B_R.
\end{equation}
Then
\begin{equation}\label{eq: section 2 imag = 0}
\IM \ints_{\pa B_\rho} u_r \bar{u} d\sigma = 0.
\end{equation}
for every $0 \leq \rho \leq R$.
\end{lemma}

\begin{proof}
This is a well-known lemma. It can be proved by multiplying \eqref{eq: section 2 helm omogenea} by $\bar{u} \chi_{B_\rho}$, integrating by parts and taking the imaginary part of the equation.
\end{proof}

The following lemma is crucial to prove $H^1$ convergence of the approximating solution.

\begin{lemma}\label{lemma seminorm norm}
Let $u$ be a solution of \eqref{eq: section 2 helm omogenea} and assume that
\begin{equation}\label{hp n 1}
\big\| 1 - \frac{n_\sharp}{n} \big\|_{L^\infty(\RR^d)} < 1,
\end{equation}
where
\begin{equation}\label{n sharp}
n_\sharp(r) = \frac{1}{|\partial B_r|} \int_{\partial B_r} u d\sigma,
\end{equation}
for every $0 \leq r \leq R$. Then we have
\begin{equation}\label{eq: L2 bounded by seminorm}
\int_{B_R} \left( |\nabla u|^2 + k^2n^2 |u|^2 \right) \frac{dx}{1+|x|} \leq \Big( 1-\big\| 1 - \frac{n_\sharp}{n} \big\|_{L^\infty(\RR^d)} \Big)^{-1} [u]_R^2  .
\end{equation}
\end{lemma}

\begin{proof}
We assert that
\begin{equation}\label{eq: proof 1}
\Big{|} 2 \IM \int_{B_R} k n u_r \bar{u} \frac{dx}{1+|x|}  \Big{|} \leq \big\| 1 - \frac{n_\sharp}{n} \big\|_{L^\infty(\RR^d)} [u]_R^2.
\end{equation}
Indeed, we write
\begin{equation} \label{eq: proof 2}
\IM \int_{B_R} k n u_r \bar{u} \frac{dx}{1+|x|} = \IM \int_{B_R} k (n-n_\sharp) u_r \bar{u} \frac{dx}{1+|x|} + \IM \int_{B_R} k n_\sharp u_r \bar{u} \frac{dx}{1+|x|};
\end{equation}
and notice that, since $n_\sharp$ is a real-valued radial function, coarea formula and \eqref{eq: section 2 imag = 0} yield
\begin{equation}\label{eq: proof 3}
\IM \int_{B_R} k n_\sharp u_r \bar{u} \frac{dx}{1+|x|} = k \int_0^R \frac{n_\sharp(\rho)}{1+\rho} \left( \IM\int_{\partial B_\rho} u_r \bar{u} d\sigma \right) d\rho = 0.
\end{equation}
Since
\begin{equation}\label{eq: proof 4}
\Big{|} \int_{B_R} 2k (n-n_\sharp) u_r \bar{u} \frac{dx}{1+|x|} \Big{|} \leq  \big\| 1 - \frac{n_\sharp}{n} \big\|_{L^\infty(\RR^d)}  \int_{B_R} \big( |u_r|^2 + k^2 n(x)^2 |u|^2 \big) \frac{dx}{1+|x|},
\end{equation}
from \eqref{eq: proof 2}-\eqref{eq: proof 4} we easily get \eqref{eq: proof 1}. From \eqref{eq: proof 1} and being
\begin{equation*}
[u]_R^2 = \ints_{B_R} \Big( |\nabla u|^2 + k^2n^2 |u|^2 \Big) \frac{dx}{1+|x|} - 2 \IM \int_{B_R} k n u_r \bar{u} \frac{dx}{1+|x|},
\end{equation*}
we obtain \eqref{eq: L2 bounded by seminorm}.
\end{proof}

In the following proposition, we prove that the minimizer of \eqref{eq: section 2 minimiz pb} exists, is unique and it is a solution of the Helmholtz equation in $B_R$.

\begin{proposition} \label{prop exist uniq minimizer}
Let $R>0$ be fixed and assume that $n$ satisfies \eqref{hp n 0} and \eqref{hp n 1}. Then, there exists a unique $u_{B_R} \in \mathcal{D}_R$ such that $[u_{B_R}]_R=d_R$, where $d_R$ is given by \eqref{eq: section 2 d_R}.
\end{proposition}

\begin{proof}
Let $\{u_m\}_{m\in\NN} \subset \mathcal{D}_R$ be a minimizing sequence for \eqref{eq: section 2 minimiz pb} such that
\begin{equation}\label{eq u_m minim seq}
[u_m]_R^2 \leq d_R^2 + \frac{1}{m},
\end{equation}
for any $m\in\NN$. The parallelogram identity yields
\begin{equation*}
\Big[ \frac{u_\ell-u_m}{2} \Big]_R^2 + \Big[ \frac{u_\ell+u_m}{2} \Big]_R^2 = \frac{[u_\ell]_R^2+[u_m]_R^2}{2},
\end{equation*}
for every $\ell \in \NN$. Since $(u_\ell+u_m)/2 \in \mathcal{D}_R$, then
\begin{equation*}
\Big[ \frac{u_\ell-u_m}{2} \Big]_R^2 \leq \frac{[u_\ell]_R^2+[u_m]_R^2}{2} - d_R^2,
\end{equation*}
and from \eqref{eq u_m minim seq} we obtain
\begin{equation*}
[u_\ell-u_m]_R^2 \leq \frac{4}{m},
\end{equation*}
for any $\ell>m$. From \eqref{eq: L2 bounded by seminorm} and \eqref{hp n 0}, we have that $\{u_m\}_{m\in\NN}$ is a Cauchy sequence in $H^1(B_R)$ and hence there exists $u_{B_R} \in H^1(B_R)$ such that $u_m \to u_{B_R}$ in $H^1(B_R)$. From such strong convergence we get that $u_{B_R}$ satisfies \eqref{eq: section 2 Helm var formul} and then $u_{B_R}\in \mathcal{D}_R$.

Now we prove the uniqueness by contradiction. Let us assume that there are two minimizers $u$ and $v$. The parallelogram identity yields
$[u-v]_R^2\leq 0$ which is a contradiction on account of Lemma \ref{lemma seminorm norm}.
\end{proof}

\begin{lemma}\label{lemma dR increasing}
Let $U$ be the solution of \eqref{eq: introd Helm Perth Vega} satisfying \eqref{eq: intro rad con Perth Vega II} and set
\begin{equation*}
[U]_\infty = \Big(\: \ints_{\RR^d} \Big| \nabla u - ikn u \frac{x}{|x|} \Big|^2  \frac{dx}{1+|x|} \Big)^{\frac{1}{2}}.
\end{equation*}
The sequence $\{d_R\}_{\{R>0\}}$ is nondecreasing and $d_R \leq [U]_\infty$, for every $R>0$.
\end{lemma}

\begin{proof}
Let $R$ be fixed and consider $\rho>0$. Since $u_{B_{R+\rho}} \big|_{B_R} \in \DD_R$, we have
\begin{equation*}
d_R \leq [u_{B_{R+\rho}}]_R \leq d_{R+\rho};
\end{equation*}
hence, $d_R$ is non-decreasing in $R$. Since $U$ is a solution of
the Helmholtz equation in $\RR^d$, then $U \big|_{B_R} \in
\DD_R$, and thus $d_R\leq [U]_R \leq [U]_\infty$, for any $R>0$.
\end{proof}

We are ready to prove our main theorem.

\begin{theorem} \label{thm convergence}
Let $R_0>0$ be fixed and assume that $n$ satisfies \eqref{hp n 0} and \eqref{hp n 1}. Then
\begin{equation*}
\lim_{R\to +\infty} \| u_{B_R} - U \|_{H^1(B_{R_0})} = 0.
\end{equation*}
\end{theorem}

\begin{proof}
Let $\rho> R > R_0$; since
$u_{B_\rho}\big|_{B_R} \in \DD_R$, we have
\begin{equation*}
\Big[ \frac{u_{B_R}-u_{B_\rho}}{2} \Big]_R^2 + \Big[ \frac{u_{B_R}+u_{B_\rho}}{2}
\Big]_R^2 = \frac{[u_{B_R}]_R^2 + [u_{B_\rho}]_R^2}{2}\ .
\end{equation*}
Since $\DD_R$ is a convex set,
\begin{equation*}
\Big[ \frac{u_{B_R}+u_{B_\rho}}{2} \Big]_{R} \geq d_R,
\end{equation*}
and, being $[u_{B_{\rho}}]_R \leq [u_{B_{\rho}}]_\rho=d_\rho$, we have
\begin{equation*}
[u_{B_R}-u_{B_\rho}]_{R}^2 \leq 2 (d_\rho^2 - d_R^2).
\end{equation*}
Lemma \ref{lemma seminorm norm} and \eqref{hp n 0} yield
\begin{equation*}
\int_{B_R} \left( |\nabla (u_{B_R}-u_{B_\rho})|^2 + k^2 n_*^2 |u_{B_R}-u_{B_\rho}|^2 \right) \frac{dx}{1+|x|}  \leq 2 (d_\rho^2 - d_R^2),
\end{equation*}
and, since $R_0<R$,
\begin{equation*}
 \| u_{B_R}-u_{B_\rho} \|_{H^1(B_{R_0})}^2 \leq \frac{2}{\min(1,k^2n_*^2)} (d_\rho^2 - d_R^2).
\end{equation*}
We notice that Lemma \ref{lemma dR increasing} implies that $d_R$ converges to a finite number as $R\to +\infty$ and hence $\{u_{B_R}\}_{R>R_0}$ is a Cauchy sequence in $H^1(B_{R_0})$. Thus $u_{B_R}$ converges strongly in $H^1(B_{R_0})$ and the limit is a solution of \eqref{eq: section 2 Helm var formul} in $B_{R_0}$. Since $R_0$ is arbitrary, we can define a solution $\tilde{u}$ of \eqref{eq: section 2 Helm var formul} in $\RR^d$ which satisfies \eqref{eq: intro rad con Perth Vega II}. From the uniqueness of the solution, we get that $\tilde{u}=U$ and the proof is complete.
\end{proof}

In the rest of this section, we show how to generalize Theorem \ref{thm convergence} to problems in exterior domains. The generalization is straightforward but it needs some slight modifications as we are going to show in the following.

Let $D \subset \RR^d$ be a bounded piecewise regular domain containing the origin and consider the following problems:
\begin{equation}\label{Probl exterior domain}
\begin{cases}
\Delta u + k^2n(x)^2 u = f, & \textmd{in } \RR^d \setminus \overline{D},\\
ju+(1-j)u_\nu = g, & \textmd{on } \partial D, \\
\displaystyle \int_{\RR^d \setminus D} \Big{|} \nabla u - ikn u \frac{x}{|x|}\Big{|}^2 \frac{dx}{1+|x|} < + \infty, &
\end{cases}
\end{equation}
for $j=0,1$, $g \in L^2 (\partial D)$ and where $\nu$ is the outward normal to $\partial D$. We denote by $U_{D}^j$ the solution of problem \eqref{Probl exterior domain} (as before, we shall not specify the assumptions on $n$ to get existence and uniqueness of the solution).

For $R> \textmd{diam}\: D$, we consider the problems
\begin{multline}\label{Pb minimize exterior domain}
\inf \Bigg\{\int_{B_R \setminus D} \big{|} \nabla u - ikn u x/|x| \big{|}^2 \frac{dx}{1+|x|}:\\ \Delta u + k^2n(x)^2 u = f,  \textmd{ in } B_R \setminus \overline{D}, \ ju+(1-j)u_\nu = g,  \textmd{ on } \partial D  \Bigg\},  \end{multline}
$j=0,1$.
Let $n_{\sharp,D}$ be given by
\begin{equation}\label{n diesis D}
n_{\sharp,D} (r) = \begin{cases} \displaystyle
\frac{1}{|\partial B_r \cap \bar{D}^c|} \int_{\partial B_r \cap \bar{D}^c} n d\sigma,  & \textmd{if } \partial B_r \cap \bar{D}^c \neq \emptyset, \\
0 & \textmd{if } \partial B_r \cap \bar{D}^c = \emptyset,\\
\end{cases}
\end{equation}
$r \geq 0$, where $\bar{D}^c = \RR^d \setminus \bar{D}$. We have the following result.

\begin{theorem}
Let $R> \textmd{diam}\: D$ be fixed and assume that $n$ satisfies \eqref{hp n 0} and \eqref{hp n 1}, where $n_\sharp$ is replaced by $n_{\sharp,D}$. Then there exists a unique minimizer $u_{D,R}^j$ of \eqref{Pb minimize exterior domain}. Furthermore $u_{D,R}^j$ satisfies $\Delta u + k^2n(x)^2 u = f$ in $B_R \setminus \bar{D}$ and we have
\begin{equation*}
\lim_{R\to +\infty} \| u_{D,R}^j - U_D^j \|_{H^1(B_{R_0} \setminus \bar{D})} = 0,
\end{equation*}
for any $R_0> \textmd{diam } D$.
\end{theorem}

\begin{proof}
The proof follows the lines of the one of Theorem \ref{thm convergence}. However, since we are considering an exterior problem, Lemmas \ref{lemma imaginary part}-\ref{lemma seminorm norm} must be adapted to this case.

\noindent Let $u$ be a solution of
\begin{equation*}
\begin{cases}
\Delta u + k^2 n(x)^2 u = 0 , & \textmd{in } B_R \setminus \overline{D},\\
ju+(1-j)u_\nu = 0, & \textmd{on } \partial D.
\end{cases}
\end{equation*}
Let $R_*$ be the radius of the largest ball contained in $D$ centered at the origin. For $R_*<r\leq R$, we consider the domain
\begin{equation*}
\Omega_r=B_r \setminus \overline{D}.
\end{equation*}
If $r>\textmd{diam}\: D$, then it is easy to show that
\begin{equation*}
\IM \int_{\partial B_r} u_r \bar{u} d\sigma = 0.
\end{equation*}
If $R_*<r\leq \textmd{diam}\: D$, then the assumptions on $D$ guarantee that the divergence theorem holds on any connected component of $\Omega_r$. From Helmholtz equation and the homogeneous boundary conditions on $\partial D$, we obtain that
\begin{equation} \label{eq: imag part 0 ext}
\IM \int_{\partial B_r \cap \bar{D}^c} u_r \bar{u} d\sigma = 0,
\end{equation}
for $R_*\leq r \leq R$. Hence, we have proven the analogous of Lemma \ref{lemma imaginary part} for the exterior problem. To prove the analogous of Lemma \ref{lemma seminorm norm} we have to take care of the application of the coarea formula in \eqref{eq: proof 3}. This can be easily done by extending $u$ and $\nabla u$ to zero in $D$. The rest of the argument needed to prove this theorem is completely analogous to what done before for the case $D = \emptyset$.
\end{proof}

\section{Numerical results for a constant index of refraction} \label{section 3}
In this section, we study some convergence properties of the approximating solution given by the minimizer $u_\Omega$ of \eqref{eq: introd minimiz problem PV}. Our main goal in this section is to study the convergence for several choices of the shape of the artificial boundary in the case of constant index of refraction.

We consider a classical model problem in scattering theory for which an exact solution is available.
 In particular, we consider
\begin{equation} \label{eq: sect4 scattering}
\begin{cases}
\Delta u + k^2 u = 0, & \textmd{in } \RR^2 \setminus \overline{B_{\hat{r}}},\\
u = g, & \textmd{on } \pa B_{\hat{r}},\\
\int_{\RR^2 \setminus B_{\hat{r}}} \big{|} \nabla u -ik u x/|x| \big{|}^2 dx < +\infty , &
\end{cases}
\end{equation}
where $B_{\hat{r}}$ is the ball of radius $\hat{r}$ centered at the origin and $g \in L^2 (\partial B_{\hat{r}})$ is a given function.
 We notice that, since $n$ is constant, the radiation condition at infinity that we use in \eqref{eq: sect4 scattering} is equivalent
to the usual Sommerfeld radiation condition.
This can be easily seen by expanding the solution at infinity in terms of negative powers of $|x|$.
We stress that in this case we use a functional without the weight $(1+|x|)^{-1}$ because we expect
a better rate of convergence for the approximate solution. It is straightforward to verify that the results in Section \ref{section 2} hold true also without such weight.

As it is well-known, the exact solution to problem \eqref{eq: sect4 scattering} can be obtained by separation of variables and is given by
\begin{equation*}
u_{exact}(\theta,r)=\frac{1}{\pi}\hat{\sum}^\infty_{\ell=0} \int^{2\pi}_0 \frac{H_\ell^1(kr)}{H_\ell^1(k\hat{r})}\cos [\ell(\theta-\tilde{\theta})] g(\tilde{\theta})d\tilde{\theta},
\end{equation*}
where $(\theta,r)$ denote the usual polar coordinates, $H_\ell^1$ is the Hankel function of the first kind of order $\ell$ and $\hat{\sum}$ means that the term with $\ell=0$ is multiplied by a factor 1/2.

Let $\Omega$ be a bounded domain containing $B_{\hat{r}}$.
For a function $f \in L^2(\partial \Omega)$ we consider the following problem: to find $u_{\Omega}[f]$ such that
\begin{equation*}P1:
\begin{cases}
\Delta u_\Omega[f] + k^2 u_\Omega[f] = 0, \qquad \textmd{in } \Omega\backslash B_{\hat{r}} \\
u_\Omega[f] = g,\quad \textmd{on } \pa B_{\hat{r}} \\
C_Du_\Omega[f] + C_N\partial_\nu u_\Omega[f] = f,\quad \textmd{on } \pa \Omega,
\end{cases}
\end{equation*}
where $C_D$ and $C_N$ may take the value 0 or 1 according to the choice of the boundary condition
to impose in $\pa \Omega$ (Dirichlet or Neumann boundary condition). The function $f$ is the unknown of the problem and it must be chosen so that the corresponding solution  $u_\Omega[f]$ minimizes the integral

\begin{equation}\label{integrale_num}
J(u_\Omega[f])=\ints_{\Omega\backslash B_{\hat{r}} } \Big| \nabla u_\Omega[f] (x) - iku_\Omega[f] (x) \frac{x}{|x|}\Big|^2dx.
\end{equation}

We solve numerically the problem by using a (CGM) conjugate gradient method (see for instance \cite{At}). In particular, we need to evaluate $J^{'}$ (the Fr\'echet derivative of $J$), and then consider the following two problems:

\begin{equation*}\hspace*{-0.4cm}P2:
\begin{cases}
\Delta p + k^2 p = -J^{'}, \qquad \textmd{in } \Omega\backslash B_{\hat{r}} ,\\
p = 0,\quad \textmd{on } \pa B_{\hat{r}},\\
C_Dp + C_N\partial_\nu p = 0,\quad \textmd{on } \pa \Omega,
\end{cases}P3:
\begin{cases}
\Delta \delta [u] + k^2 \delta [u] = 0, \qquad \textmd{in } \Omega\backslash B_{\hat{r}} ,\\
\delta [u] = 0,\quad \textmd{on } \pa B_{\hat{r}},\\
C_D\delta [u] + C_N\partial_\nu \delta [u] = w_0,\quad \textmd{on } \pa \Omega.
\end{cases}
\end{equation*}

Problem $P2$ is the adjoint problem to Problem $P1$ while Problem $P3$ computes the small
variation $\delta [u]$ of the solution during the iterative process of the CGM algorithm which is exposed below:
\begin{algorithm}
\caption{Conjugate gradient algorithm}
\begin{algorithmic}[1]
\REQUIRE initial guess $f=f_0$, tolerance $\epsilon$
\STATE Solve $P1$
\STATE Evaluate $J^{'}(u_\Omega[f])$
\STATE Solve P2.
\STATE Set $g_0:=p$,$w_0:=p,m=1$
\WHILE {stopping criteria $\gamma<\epsilon$ is not satisfied}
\STATE Solve $P3$
\STATE Evaluate $J^{'}(\delta [u])$
\STATE Solve $P2$
\STATE Set: $\rho=<g_{m-1},w_{m-1}>_{L^2}/<(C_Np+C_D\partial_ {\nu}p),w_{m-1}>_{L^2}$,\\
 $g_m=g_{m-1}-\rho (C_Np+C_D\partial_{\nu}p) $
\STATE Set: $u_\Omega[f]=u_\Omega[f]-\rho \delta [u],\quad f=f_0-\rho w_{m-1}$
\STATE Set: $\gamma=<g_{m},w_{m-1}>_{L^2}/\|w_{m-1}\|_{L^2}^2$
\STATE Set: $w_m=g_{m}+\gamma w_{m-1}$
\STATE $m:=m+1$
\ENDWHILE
\end{algorithmic}
\end{algorithm}

To solve Problems $P1,P2,P3$ we use a finite elements formulation; numerical simulations are performed through FREEFEM$^{\textmd{\textcopyright}}$
software. The
finest grid used contains up to $2\cdot10^5$ grid points, and we adapt the grid close to the artificial boundaries and
where the large gradients of the solution form. Each simulation is firstly performed in a coarser grid and then the grid is gradually refined until a
fully grid-independence is reached in the numerical results. We also have checked in which way the choice of the Dirichlet
 or Neumann condition on $\partial{\Omega}$ affects the numerical results, and
 we have verified that this choice does not influence the final result in the case of smooth domains.
However, Neumann boundary conditions are in general to prefer when one deals with coarse grid, as spurious oscillation are
likely to manifest close to the boundary in the case of Dirichlet boundary condition.
For this reason, we present all the result by imposing Neumann condition on $\partial \Omega$ ($C_N=1,C_D=0$).

We shall present numerical results for three choices of the artificial domain $\Omega$: a ball, an ellipse and a square. In each simulation $\epsilon=10^{-8}$ is the tolerance in Algorithm 1 and the Dirichlet boundary condition on $\pa B_{\hat{r}}$ is given
by $g(\theta)=\cos(j\theta)$, being $j$ an integer.  The radius of $\pa B_{\hat{r}}$ is $\hat{r}=1/2$.

The convergence properties of our method are tested through the evaluation of the usual error norms and
the relative error of the integral \eqref{integrale_num}, that is we compute:
\begin{eqnarray*}
 L^2&=&(\int_{\Lambda}\|u_\Omega[f]-u_{exact} \|^2d\Lambda)^{1/2},\\
L^2_{rel}&=&L^2/(\int_{\Lambda}\|u_{exact} \|^2d\Lambda)^{1/2},\\
H^1&=&(\int_{\Lambda}\|u_\Omega[f]-u_{exact} \|^2+\|\nabla(u_\Omega[f]-u_{exact}) \|^2d\Lambda)^{1/2},\\
H^1_{rel}&=&H^1/(\int_{\Lambda}\|u_{exact} \|^2+\|\nabla u_{exact}) \|^2d\Lambda)^{1/2},\\
\delta J_{rel}&=&J(u_\Omega[f]-u_{exact})/J(u_{exact}).
\end{eqnarray*}

We shall evaluate these errors in the whole domain $\Omega\backslash B_{1/2}$ and in the
annular domain $\Omega_{1/2}^1=B_{1}\backslash B_{1/2}$ and we check how the numerical solution converges to the exact one and how
the shape of the artificial boundary of $\Omega$ affects this convergence.
In the following paragraphs we shall present the results in the case in which $\Omega$ is a ball, an ellipse, and a square.
We have performed in each case the convergence analysis for the values $j=0,2,3$, and $k=0.5,1,2$.

\paragraph{Balls}
The first case presented is relative to the annulus $\Omega_{1/2}^R=B_{R}\backslash B_{1/2}$. The numerical
simulations are performed up to $R=8$.
In Tables \ref{tableconvDJ0} we report the errors on $\Omega^1_{1/2}$ and $\Omega_{1/2}^R$ for $R=1,2,4,8$ (the $L^2_{rel},H^1_{rel}$ errors on $\Omega^1_{1/2}$ are also shown in Figs.\ref{errorL2H1d1})a-f in lin-log scale).

\begin{table}[ht]
\tablefont{2.5mm}
\begin{center}
   \hspace*{-1.7cm} \begin{tabular}{|l|l|l||l|l|l|l|l||l|l|l|l|l||}
\multicolumn{3}{|c||}{} & \multicolumn{5}{|c||}{Constant index of refraction:} & \multicolumn{5}{|c||}{Constant index of refraction:} \\
\multicolumn{3}{|c||}{} & \multicolumn{5}{|c||}{Errors on $\Omega_{1/2}^1$ ($\cdot 10^{-1}$)} & \multicolumn{5}{|c||}{Errors on $\Omega_{1/2}^R$($\cdot 10^{-1}$)} \\ \hline
$j$ &$k$ &    $R$ &$L^{2}$ & $L^{2}_{rel}$ & $H^1$ & $H^1_{rel}$ & $\delta J_{rel}$ &$L^{2}$ & $L^{2}_{rel}$ & $H^1$ & $H^1_{rel}$ & $\delta J_{rel}$ \\ \hline

 0 &0.5 &     $1$ & $0.664$ & $0.509$ & $2.06$ & $1.21$ &$3.22$& $ $ & $ $ & $ $ & $ $ &$ $\\
 &&      $2$ & $0.341$ & $0.259$ & $1.05$& $0.615$& $1.21$& $1.61 $ & $0.679 $ & $2.06 $ & $0.714 $ &$0.859 $\\
  &&      $4$ & $0.161$ & $0.123$ & $0.498$&$0.292$&$0.854$& $1.95 $ & $0.52 $ & $2.07 $ & $0.471 $ &$0.0981 $\\
 &&      $8$ & $0.113$ & $0.0865$ & $0.352$& $0.205$ &$0.499$& $1.72 $ & $0.308 $ & $1.97 $ & $0.309 $ &$0.0228 $\\
\hline
 0 &1 &     $1$ & $0.421$ & $0.331$ & $1.19$ & $0.602$& $2.25$& $ $ & $ $ & $ $ & $ $ &$ $\\
  & &     $2$ & $0.211$ & $0.163$ & $0.642$&$0.308$& $1.17$& $0.851 $ & $0.374 $ & $1.09 $ & $0.315 $ &$0.932 $\\
& &     $4$ & $0.139$ & $0.107$ & $0.423$&$0.203$&$0.941$& $0.801 $ & $0.226 $ & $0.121 $ & $0.231 $ &$0.205 $\\
  & &     $8$ & $0.0661$ & $0.0512$ & $0.201$&$0.0967$&$0.452$& $0.589 $ & $0.112 $ & $0.841 $ & $0.111 $ &$0.0887 $\\
\hline
0 &2 &    $1$ & $0.261$ & $0.205$ & $0.767$ &$0.251$& $0.112$& $ $ & $ $ & $$ & $$ &$ $\\
 & &    $2$ & $0.138$ & $0.108$ & $0.407$& $0.133$& $0.836$& $0.318 $ & $0.143 $ & $0.699 $ & $0.135 $ &$0.0621 $\\
 &&     $4$ & $0.0788$ & $0.0621$ & $0.232$&$0.0759$& $0.485$& $0.265 $ & $0.0777 $ & $0.596 $ & $0.0805 $ &$0.00921 $\\
 & &     $8$ & $0.0178$ & $0.0141$ & $0.0526$&$0.0172$&$0.112$& $0.0933 $ & $0.0186 $ & $0.232 $ & $0.0179 $ &$0.00143 $\\
\hline
 2 &.5 &     $1$ & $1.88$ & $3.41$ & $7.77$ & $3.13$ &$0.582$& $ $ & $ $ & $ $ & $ $ &$ $\\
 &&      $2$ & $0.211$ & $0.383$ & $0.872$& $0.351$& $0.158$& $1.68$ & $2.96$ & $3.71$ & $1.43$ &$0.0447$\\
  &&      $4$ & $0.0209$ & $0.0379$ & $0.0863$&$0.0348$&$0.0654$& $1.23$ & $1.85$ & $1.58$ & $0.609$ &$0.0172$\\
 &&      $8$ & $0.00327$ & $0.00595$ & $0.00135$& $0.00545$ &$0.00168$& $0.693$ & $0.993$ & $0.753$ & $0.289$ &$0.00525$\\
\hline
 2 &1 &     $1$ & $1.86$ & $0.323$ & $7.62$ & $3.09$& $0.687$& $ $ & $ $ & $ $ & $ $ &$ $\\
  & &     $2$ & $0.245$ & $0.425$ & $1.01$&$0.407$& $0.168$& $1.48$ & $2.69$ & $3.62$ & $1.38$ &$0.0181$\\
& &     $4$ & $0.0431$ & $0.0749$ & $0.176$&$0.0716$&$0.0134$& $1.21$ & $1.49$ & $1.59$ & $0.596$ &$0.0625$\\
  & &     $8$ & $0.0218$ & $0.0379$ & $0.895$&$0.0363$&$0.107$& $0.878$ & $0.905$ & $1.21$ & $0.435$ &$0.0249$\\
\hline
2 &2 &    $1$ & $1.66$ & $2.51$ & $6.59$ &$2.61$& $0.905$& $ $ & $ $ & $ $ & $ $ &$ $\\
 & &    $2$ & $0.357$ & $0.539$ & $1.41$& $0.559$& $0.0836$& $1.59$ & $1.61$ & $3.13$ & $1.03$ &$0.401$\\
 &&     $4$ & $0.185$ & $0.281$ & $0.736$&$0.291$& $0.0282$&  $1.21$ & $0.933$ & $2.58$ & $0.715$ &$0.163$\\
 & &     $8$ & $0.102$ & $0.155$ & $0.406$&$0.161$&$0.0172$& $0.924$ & $0.519$ & $2.12$ & $0.464$ &$0.0801$\\
\hline

 3 &.5 &     $1$ & $1.56$ & $3.63$ & $8.81$ & $2.66$ &$0.0143$& $ $ & $ $ & $ $ & $ $ &$ $\\
 &&      $2$ & $0.0446$ & $0.103$ & $0.232$& $0.0757$& $0.00951$& $0.721$ & $1.618$ & $1.88$ & $0.608$ &$0.00897$\\
  &&      $4$ & $0.00113$ & $0.00263$ & $0.00591$&$0.00192$&$0.000453$& $0.252$ & $0.564$ & $0.387$ & $0.124$ &$0.000984$\\
 &&      $8$ & $0.000051$ & $0.000118$ & $0.000266$& $0.0000865$ &$0.000021$& $0.0957$ & $0.214$ & $0.108$ & $0.0348$ &$0.000064$\\
\hline
 3 &1 &     $1$ & $1.61$ & $20573$ & $8.33$ & $2.71$& $0.0561$& $ $ & $ $ & $ $ & $ $ &$ $\\
  & &     $2$ & $0.0569$ & $0.131$ & $0.294$&$0.0959$& $0.0351$& $0.822$ & $0.181$ & $0.208$ & $0.673$ &$0.00781$\\
& &     $4$ & $0.00288$ & $0.00658$ & $0.0149$&$0.004816$&$0.000396$& $0.352$ & $0.763$ & $0.511$ & $0.164$ &$0.00347$\\
  & &     $8$ & $0.000921$ & $0.00211$ & $0.00477$&$0.00155$&$0.000216$& $0.211$ & $0.452$ & $0.299$ & $0.0963$ &$0.000988$\\
\hline
3 &2 &    $1$ & $1.73$ & $3.65$ & $8.71$ &$2.83$& $0.367$& $ $ & $ $ & $ $ & $ $ &$ $\\
 & &    $2$ & $0.121$ & $0.255$ & $0.609$& $0.198$& $0.0654$& $1.09$ & $2.04$ & $2.55$ & $0.808$ &$0.127$\\
 &&     $4$ & $0.0415$ & $0.0878$ & $0.209$&$0.0681$& $0.00143$&  $0.707$ & $1.19$ & $1.58$ & $0.494$ &$0.0401$\\
 & &     $8$ & $0.0181$ & $0.0383$ & $0.0913$&$0.0297$&$0.000956$& $0.456$ & $0.661$ & $1.09$ & $0.305$ &$0.0188$\\
\hline
\end{tabular}
\caption{The $L^2,L^2_{rel},H^1_{rel},\delta J_{rel}$ errors of the numerical solution on the annulus $\Omega_{1/2}^1$ and $\Omega_{1/2}^R$ by varying the exterior radius $R$ of the annular domain $\Omega_{1/2}^R$ for the case $j=0,2,3,k=0.5,1,2$.\label{tableconvDJ0}}
\end{center}
\end{table}

 \begin{figure}
\vspace*{-1cm}\subfigure[]{\hspace*{-2cm}\includegraphics[width=8cm,height=6cm]{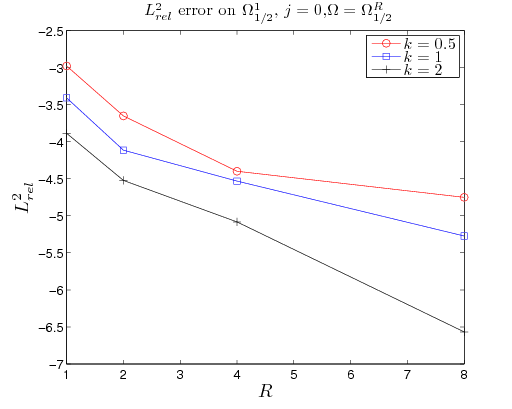}}
\subfigure[]{\hspace*{-0.5cm}\includegraphics[width=8cm,height=6cm]{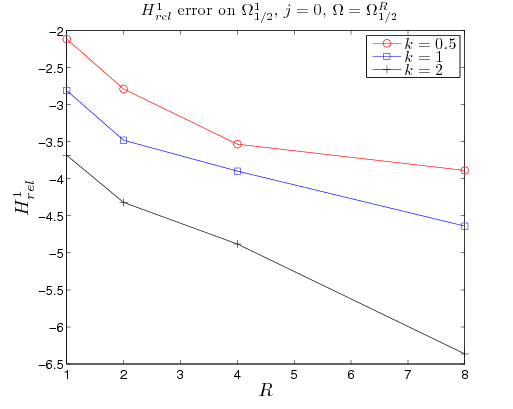}}
\subfigure[]{\hspace*{-2cm}\includegraphics[width=8cm,height=6cm]{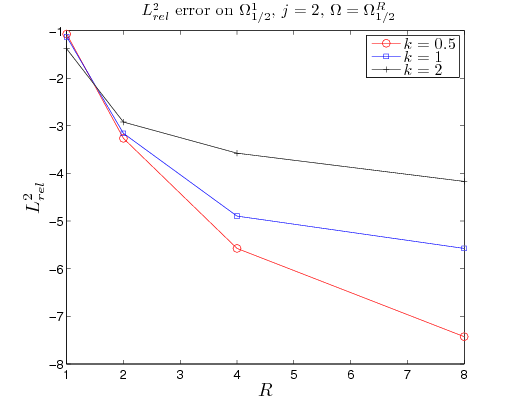}}
\subfigure[]{\hspace*{-0.5cm}\includegraphics[width=8cm,height=6cm]{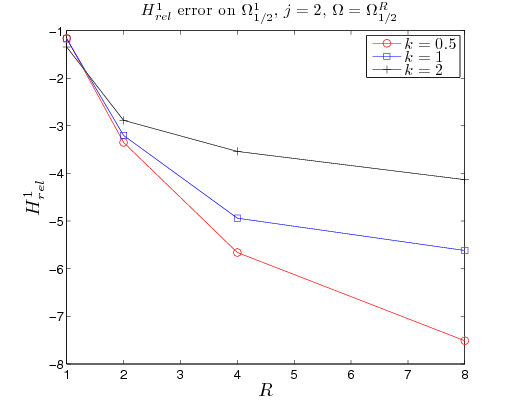}}
\subfigure[]{\hspace*{-2cm}\includegraphics[width=8cm,height=6cm]{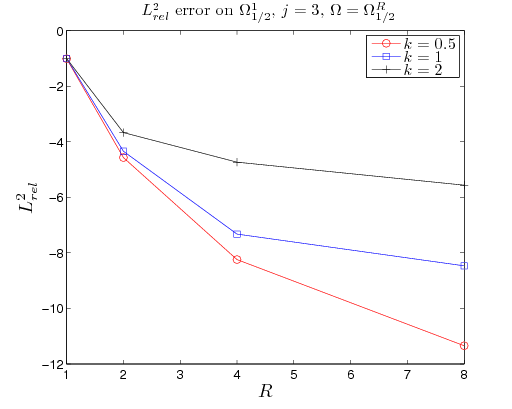}}
\subfigure[]{\hspace*{-0.5cm}\includegraphics[width=8cm,height=6cm]{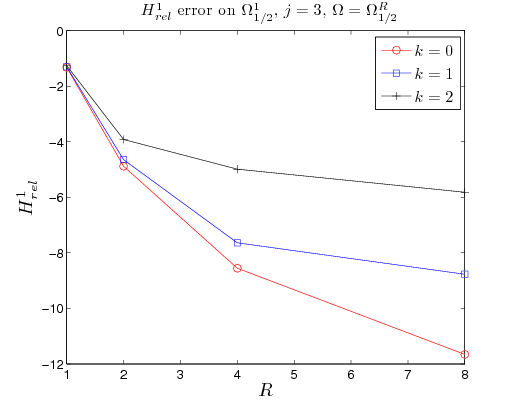}}
\caption{Numerical errors (lin-log scale) on the domain $\Omega_{1/2}^1$ by varying the exterior radius $R$ of the domain $\Omega_{1/2}^R$. }
\label{errorL2H1d1}
 \end{figure}

\paragraph{Ellipses}

The second case presented is relative to the domain $E_{1/2}^R=\mathcal{E}_{R}\backslash B_{1/2}$ where $\mathcal{E}_{R}$ is the ellipses
of equation $x^2/(2R)^2+y^2/(R)^2=1$. The numerical
simulations are performed up to $R=8$. We mention that, with respect to the previous case,
 the CGM method is slightly slower to reach the desired tolerance.
In Tables \ref{tableconvEJ0} we report the errors on $\Omega^1_{1/2}$ and $E_{1/2}^R$ for $R=1,2,4,8$ (the $L^2_{rel},H^1_{rel}$ errors on $\Omega^1_{1/2}$ are also shown in Figs.\ref{errorL2H1e1})a-f in lin-log scale).

\begin{table}[ht]
\tablefont{2.5mm}
\begin{center}
   \hspace*{-1.7cm} \begin{tabular}{|l|l|l||l|l|l|l|l||l|l|l|l|l||}
\multicolumn{3}{|c||}{} & \multicolumn{5}{|c||}{Constant index of refraction:} & \multicolumn{5}{|c||}{Constant index of refraction:} \\
\multicolumn{3}{|c||}{} & \multicolumn{5}{|c||}{Errors on $\Omega_{1/2}^1$ ($\cdot 10^{-1}$)} & \multicolumn{5}{|c||}{Errors on $E_{1/2}^R$($\cdot 10^{-1}$)} \\ \hline
$j$ &$k$ &    $R$ &$L^{2}$ & $L^{2}_{rel}$ & $H^1$ & $H^1_{rel}$ & $\delta J_{rel}$ &$L^{2}$ & $L^{2}_{rel}$ & $H^1$ & $H^1_{rel}$ & $\delta J_{rel}$ \\ \hline

 0 &0.5 &     $1$ & $0.572$ & $0.435$ & $1.781$ & $1.04$ &$2.11$& $1.27 $ & $0.711 $ & $2.31 $ & $1.03 $ &$2.39 $\\
 &&      $2$ & $0.291$ & $0.221$ & $0.902$& $0.528$& $0.966$& $2.24 $ & $0.761 $ & $2.646 $ & $0.751 $ &$0.516 $\\
  &&      $4$ & $0.142$ & $0.108$ & $0.441$ & $0.258$ & $0.49$ & $2.27$ & $0.503$ & $2.51$ & $0.478$  & $0.011 $\\
 &&      $8$ & $0.0798$ & $0.0607$ & $0.245$& $0.144$ &$0.581$& $1.54 $ & $0.232 $ & $1.731 $ & $0.229 $ &$0.0098 $\\
\hline
 0 &1 &     $1$ & $0.382$ & $0.296$ & $1.19$ & $0.574$& $2.78$& $0.0807 $ & $0.0467 $ & $1.55 $ & $0.576 $ &$0.721 $\\
  & &     $2$ & $0.179$ & $0.139$ & $0.553$&$0.265$& $1.15$& $1.061 $ & $0.379 $ & $1.45 $ & $0.345 $ &$0.12$\\
& &     $4$ & $0.0905$ & $0.0701$ & $0.275$&$0.132$&$0.641$& $0.671 $ & $0.159 $ & $0.967 $ & $0.156 $ &$0.095 $\\
  & &     $8$ & $0.0496$ & $0.0384$ & $0.151$&$0.0726$&$0.231$& $0.549 $ & $0.0881 $ & $0.76 $ & $0.0863 $ &$0.052 $\\
\hline
0 &2 &    $1$ & $0.219$ & $0.172$ & $0.689$ &$0.225$& $0.432$& $0.405 $ & $0.241 $ & $0.944$ & $0.237$ &$0.192 $\\
 & &    $2$ & $0.115$ & $0.0907$ & $0.342$& $0.112$& $0.211$& $0.332 $ & $0.122 $ & $0.749 $ & $0.119 $ &$0.0921 $\\
 &&     $4$ & $0.0585$ & $0.0461$ & $0.173$&$0.0523$& $0.115$& $0.249 $ & $0.0611 $ & $0.559 $ & $0.0603 $ &$0.0531 $\\
 & &     $8$ & $0.00601$ & $0.00474$ & $0.00178$&$0.00582$&$0.0371$& $0.045 $ & $0.00772 $ & $0.101 $ & $0.0076 $ &$0.00821 $\\
\hline

 2 &0.5 &     $1$ & $1.282$ & $2.232$ & $5.12$ & $2.051$ &$0.449$& $2.18 $ & $3.71 $ & $6.42 $ & $2.53 $ &$0.252 $\\
 &&      $2$ & $0.161$ & $0.292$ & $0.548$& $0.221$& $0.0541$& $1.52 $ & $2.35 $ & $2.37 $ & $0.921 $ &$0.0471 $\\
  &&      $4$ & $0.0239$ & $0.0435$ & $0.0795$&$0.0321$&$0.00249$ & $1.25$ & $1.81$ & $1.51$ & $0.579$ & $0.0131 $\\
 &&      $8$ & $0.00964$ & $0.0175$ & $0.0305$& $0.0123$ &$0.00125$& $0.715 $ & $0.991 $ & $0.794 $ & $0.303 $ &$0.0032 $\\
\hline
 2 &1 &     $1$ & $1.27$ & $2.21$ & $5.05$ & $2.05$& $0.525$& $2.21 $ & $3.47 $ & $6.45 $ & $2.54 $ &$0.411 $\\
  & &     $2$ & $0.0186$ & $0.324$ & $0.679$&$0.275$& $0.072$& $1.97 $ & $2.61 $ & $3.31 $ & $1.25 $ &$0.145$\\
& &     $4$ & $0.0371$ & $0.0658$ & $0.133$&$0.0541$&$0.0211$& $0.844 $ & $0.951 $ & $1.17 $ & $0.432 $ &$0.0621 $\\
  & &     $8$ & $0.0261$ & $0.0451$ & $0.0921$&$0.0372$&$0.00925$& $0.842 $ & $0.775 $ & $1.09 $ & $0.32 $ &$0.0211 $\\
\hline
2 &2 &    $1$ & $1.14$ & $1.73$ & $4.51$ &$1.78$& $0.634$& $2.04 $ & $2.54 $ & $6.22$ & $2.25$ &$0.757 $\\
 & &    $2$ & $0.325$ & $0.491$ & $1.24$& $0.492$& $0.0211$& $1.76$ & $1.58$ & $3.81$ & $1.16$ &$0.301 $\\
 &&     $4$ & $0.168$ & $0.253$ & $0.636$&$0.251$& $0.0132$& $1.19$ & $0.781 $ & $2.66 $ & $0.652 $ &$0.122 $\\
 & &     $8$ & $0.00182$ & $0.00275$ & $0.00711$&$0.00279$&$0.0026$& $0.0291 $ & $0.0141$ & $0.0661 $ & $0.0127 $ &$0.0089 $\\
\hline

 3 &0.5 &     $1$ & $0.942$ & $2.18$ & $4.51$ & $1.46$ &$0.0831$& $1.44 $ & $3.28 $ & $5.83 $ & $1.83 $ &$0.0431 $\\
 &&      $2$ & $0.0361$ & $0.0841$ & $0.127$& $0.0412$& $0.00354$& $0.394 $ & $0.884 $ & $0.729 $ & $0.235 $ &$0.00211 $\\
  &&      $4$ & $0.00251$ & $0.00581$ & $0.00849$&$0.00272$&$0.00249$ & $0.145$ & $0.326$ & $0.178$ & $0.0575$ & $0.000221$\\
 &&      $8$ & $0.000482$ & $0.00112$ & $0.00163$& $0.00053$ &$0.00000647$& $0.0895 $ & $0.201 $ & $0.101 $ & $0.0324 $ &$0.0000921 $\\
\hline
 3 &1 &     $1$ & $0.887$ & $2.02$ & $4.19$ & $1.36$& $0.141$& $1.42 $ & $3.18 $ & $5.57 $ & $1.81 $ &$0.0312 $\\
  & &     $2$ & $0.0584$ & $0.133$ & $0.217$&$0.0707$& $0.0112$& $0.774 $ & $1.68 $ & $1.56 $ & $0.503 $ &$0.00778$\\
& &     $4$ & $0.00471$ & $0.0107$ & $0.0168$&$0.0054$&$0.00127$& $0.281 $ & $0.603 $ & $0.394 $ & $0.126 $ &$0.00197 $\\
  & &     $8$ & $0.00265$ & $0.00605$ & $0.00921$&$0.00299$&$0.000201$& $0.185 $ & $0.392 $ & $0.258 $ & $0.0831 $ &$0.000705 $\\
\hline
3 &2 &    $1$ & $1.02$ & $2.16$ & $4.98$ &$1.62$& $0.321$& $1.82 $ & $3.64 $ & $7.03$ & $2.25$ &$0.239 $\\
 & &    $2$ & $1.04$ & $0.221$ & $0.461$& $0.151$& $0.0201$& $1.17$ & $2.09$ & $2.71$ & $0.851$ &$0.0941 $\\
 &&     $4$ & $0.0415$ & $0.0878$ & $0.178$&$0.0581$& $0.00201$& $0.638$ & $1.07 $ & $1.41 $ & $0.431 $ &$0.0301 $\\
 & &     $8$ & $0.000811$ & $0.0171$ & $0.0272$&$0.000881$&$0.00000286$& $0.0178 $ & $0.0237$ & $0.0401 $ & $0.0118 $ &$0.00168 $\\
\hline
\end{tabular}
\caption{The $L^2,L^2_{rel},H^1_{rel},\delta J_{rel}$ errors of the numerical solution on  $\Omega_{1/2}^1$ and $E_{1/2}^R$ by varying the semi-minor axis $R$ of $E_{1/2}^R$ for the case $j=0,2,3,k=0.5,1,2$.\label{tableconvEJ0}}
\end{center}
\end{table}

 \begin{figure}
\vspace*{-1cm}\subfigure[]{\hspace*{-2cm}\includegraphics[width=8cm,height=6cm]{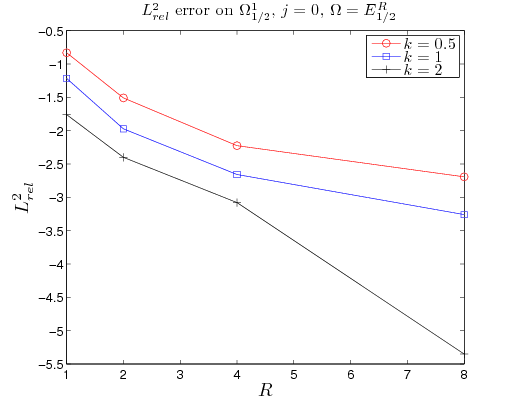}}
\subfigure[]{\includegraphics[width=8cm,height=6cm]{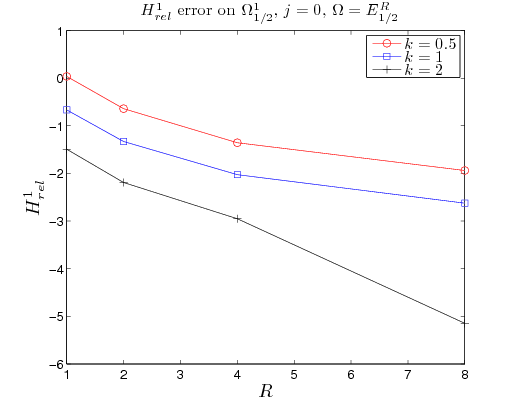}}
\subfigure[]{\hspace*{-2cm}\includegraphics[width=8cm,height=6cm]{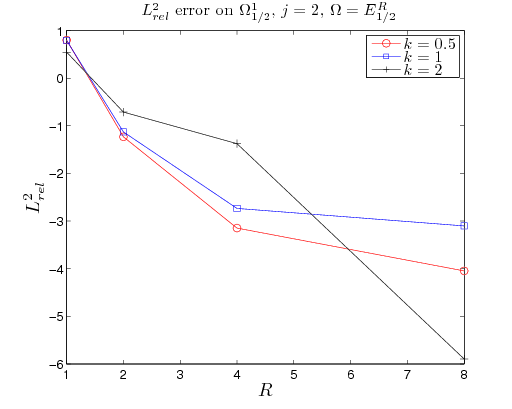}}
\subfigure[]{\includegraphics[width=8cm,height=6cm]{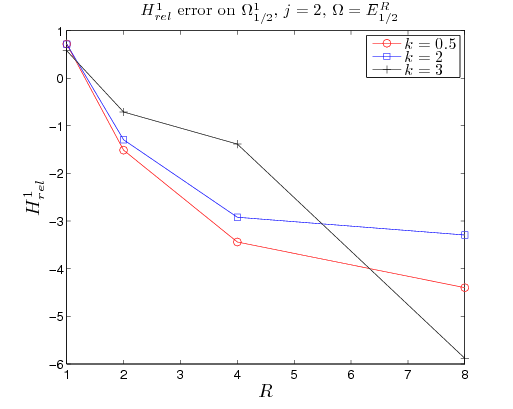}}
\subfigure[]{\hspace*{-2cm}\includegraphics[width=8cm,height=6cm]{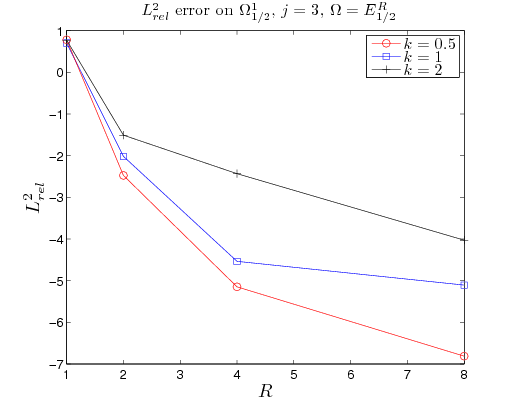}}
\subfigure[]{\includegraphics[width=8cm,height=6cm]{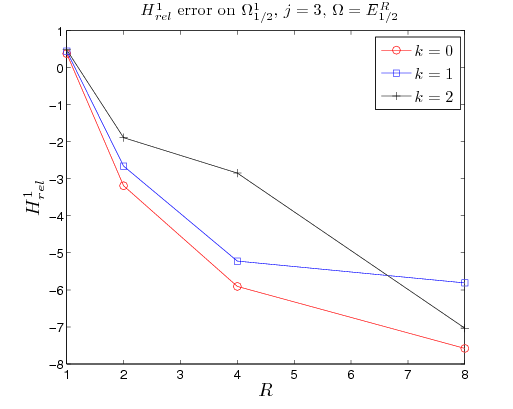}}
\caption{Numerical errors (lin-log scale) on the domain $\Omega_{1/2}^1$ by varying the semi-minor axis $R$ of $E_{1/2}^R$. }
\label{errorL2H1e1}
 \end{figure}

\paragraph{Squares}
We conclude this section by showing a numerical investigation on the domain $S_{1/2}^R=Q_{R}\backslash B_{1/2}$ where $Q_{R}$ is the square centered
in (0,0) and length side $2R$. The numerical
simulations are performed up to $R=8$. Even if our approach can be applied to any domain, we mention
that the presence of corners or edges deteriorate the convergence properties of the CGM method,
at least if compared to the case in which the artificial boundary is regular (circle or ellipses). We also notice that imposing the Dirichlet boundary condition on the problem P1 (i.e. $C_D=1,C_N=0$) gives rise to a relevant oscillatory behavior of the solution near the corners due to the lack of regularity of the domain, while Neumann boundary gives a more regular behavior in the numerical solution.
In Tables \ref{tableconvQJ0} we report the errors on $\Omega^1_{1/2}$ and $S_{1/2}^R$ by varying the
 semi-length side $R$ of $Q_R$ for $R=1,2,4,8$
(the $L^2_{rel},H^1_{rel}$ errors on $\Omega^1_{1/2}$ are also shown in Figs.\ref{errorL2H1s1})a-f in lin-log scale).

\begin{table}[ht]
\tablefont{2.5mm}
\begin{center}
   \hspace*{-1.9cm} \begin{tabular}{|l|l|l||l|l|l|l|l||l|l|l|l|l||}
\multicolumn{3}{|c||}{} & \multicolumn{5}{|c||}{Constant index of refraction:} & \multicolumn{5}{|c||}{Constant index of refraction:} \\
\multicolumn{3}{|c||}{} & \multicolumn{5}{|c||}{Errors on $\Omega_{1/2}^1$ ($\cdot 10^{-1}$)} & \multicolumn{5}{|c||}{Errors on $S_{1/2}^R$($\cdot 10^{-1}$)} \\ \hline
$j$ &$k$ &    $R$ &$L^{2}$ & $L^{2}_{rel}$ & $H^1$ & $H^1_{rel}$ & $\delta J_{rel}$ &$L^{2}$ & $L^{2}_{rel}$ & $H^1$ & $H^1_{rel}$ & $\delta J_{rel}$ \\ \hline

 0 &0.5 &     $1$ &0.631&0.481&1.95&1.14&2.81&0.879&0.594&2.14&1.13&2.82 \\
 &&      $2$ & 0.321&0.243&0.986&0.577&0.571&1.84&0.717&2.269&0.727&0.645\\
  &&      $4$ & 0.153&0.116&0.471&0.276&0.231&2.07&0.516&2.22&0.476&0.0810\\
 &&      $8$ & 0.101&0.0771&0.312&0.182&0.0721&1.65&0.278&1.88&0.277&0.0451\\
\hline
 0 &1 &     $1$ & 0.0421&0.327&1.302&0.625&0.806&0.581&0.402&1.431&0.621&1.01\\
  & &     $2$ & 0.197&0.153&0.603&0.289&0.606&0.944&0.381&1.23&0.331&0.606\\
& &     $4$ & 0.129&0.101&3.93&0.181&0.431&0.701&0.201&1.181&0.212&0.0151\\
  & &     $8$ & 0.0467&0.0362&0.142&0.0681&0.0832&0.489&0.0881&0.684&0.0854&0.00521 \\
\hline
 0 &2 &     $1$ & 0.241 &0.188&0.729&0.238&1.935&0.317&0.223&0.798&0.235&0.0254 \\
 &&      $2$ & 0.139&0.109&0.411&0.134&0.851&0.336&0.141&0.798&0.143&0.0105\\
  &&      $4$ & 0.0497&0.0391&0.146&0.0479&0.321&0.207&0.0561&0.473&0.0571&0.0526\\
 &&      $8$ & 0.0320&0.0251&0.03942&0.0308&0.0851&0.176&0.0332&0.395&0.0322&0.0212\\
\hline

 2 &0.5 &     $1$& 1.745& 3.17&7.26&2.92&0.594&2.01&3.59&8.53&3.41&0.0297\\
 &&      $2$ & 0.193&0.351&0.799&0.322&0.143&1.91&3.02&3.96&1.53&0.0143\\
  &&      $4$ & 0.0192&0.0351&0.0796&0.0321&0.0675&1.25&1.88&1.65&0.636&0.00954\\
 &&      $8$ & 0.00315&0.00572&0.0131&0.00525&0.0016&0.695&0.993&0.764&0.293&0.00438\\
\hline
 2 &1 &     $1$ &1.73 &3.07&7.14&2.89&0.681&1.98&3.39&8.39&3.35&0.129 \\
  & &     $2$ &0.227&0.394&0.931&0.377&0.149&1.94&2.74&3.91&1.491&0.105\\
& &     $4$ & 0.0415&0.072&0.171&0.0691&0.0133&1.22&1.48&1.66&0.618&0.0441\\
  & &     $8$ & 0.0141&0.0245&0.0578&0.0234&0.0079&0.654&0.663&0.904&0.323&0.014 \\
\hline
 2 &2 &     $1$ & 1.51&2.37&6.25&2.47&0.912&1.76&2.587&7.08&2.71&0.574\\
  & &     $2$ & 0.344&0.52&1.365&0.539&0.0624&1.62&1.67&3.37&1.10&0.355\\
& &     $4$ & 0.139&0.211&0.553&0.218&0.0144&1.02&0.771&2.19&0.59&0.125\\
  & &     $8$ & 0.0961&0.144&0.381&0.151&0.0852&0.891&0.492&2.01&0.435&0.0788 \\
\hline

 3 &0.5 &     $1$ &1.17&2.71&6.13&1.99&0.97&1.63&3.75&7.56&2.45&0.107\\
 &&      $2$ & 0.0422&0.0981&0.188&0.0611&0.0132&0.704&1.58&1.69&0.546&0.00874\\
  &&      $4$ & 0.00237&0.00551&0.00861&0.00279&0.00292&0.247&0.554&0.393&0.115&0.000123\\
 &&      $8$ & 0.000191&0.000445&0.000664&0.000216&0.000112&0.0962&0.215&0.109&0.0351&0.0000232\\
\hline
 3 &1 &     $1$ & 1.21&2.76&6.31&2.05&0.148&1.69&3.81&7.76&2.51&0.042 \\
  & &     $2$ &0.0504&0.115&0.234&0.0761&0.0201&0.824&1.81&1.92&0.621&0.00602\\
& &     $4$ & 0.00304&0.00695&0.0137&0.00445&0.00986&0.356&0.771&0.521&0.167&0.00281\\
  & &     $8$ & 0.00108&0.00246&0.00437&0.00142&0.000146&0.172&0.367&0.241&0.0777&0.000712\\
\hline
 3 &2 &     $1$ &0.541&1.14&2.73&0.891&0.337&0.929&1.91&3.42&1.21&0.331\\
  & &     $2$ & 0.0988&0.208&0.497&0.162&0.0401&1.12&2.07&2.63&0.833&0.105\\
& &     $4$ & 0.0276&0.0583&0.138&0.0451&0.000102&0.501&0.931&1.25&0.392&0.189\\
  & &     $8$ & 0.0187&0.0396&0.0923&0.0301&0.000405&0.478&0.673&1.07&0.323&0.00782 \\
\hline
\end{tabular}
\caption{The $L^2,L^2_{rel},H^1_{rel},\delta J_{rel}$ errors of the numerical solution on  $\Omega_{1/2}^1$ and $S_{1/2}^R$ by varying the semi-lenght side $R$ of $S_{1/2}^R$ for the case $j=0,2,3,k=0.5,1,2$.\label{tableconvQJ0}}
\end{center}
\end{table}

 \begin{figure}
\vspace*{-1cm}\subfigure[]{\hspace*{-2cm}\includegraphics[width=8cm,height=6cm]{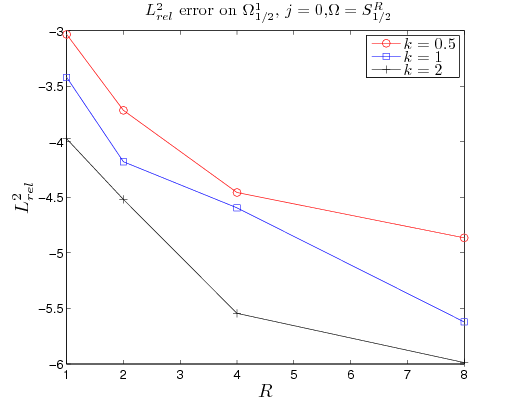}}
\subfigure[]{\includegraphics[width=8cm,height=6cm]{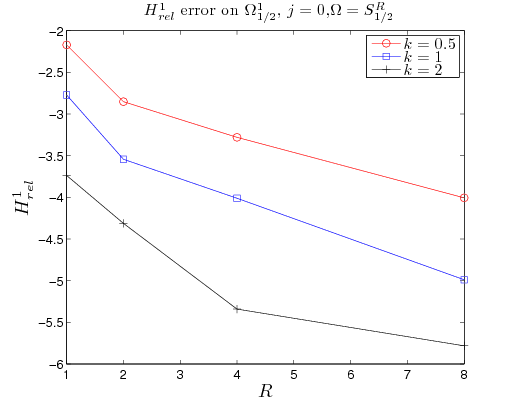}}
\subfigure[]{\hspace*{-2cm}\includegraphics[width=8cm,height=6cm]{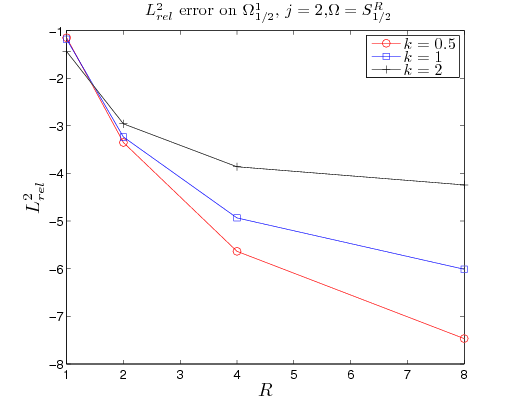}}
\subfigure[]{\includegraphics[width=8cm,height=6cm]{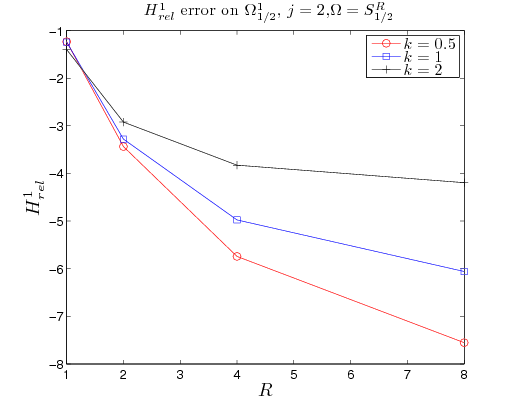}}
\subfigure[]{\hspace*{-2cm}\includegraphics[width=8cm,height=6cm]{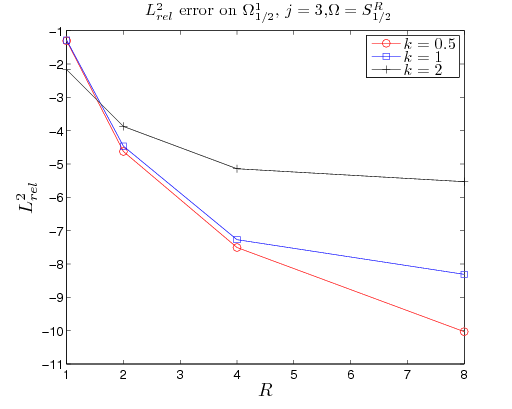}}
\subfigure[]{\includegraphics[width=8cm,height=6cm]{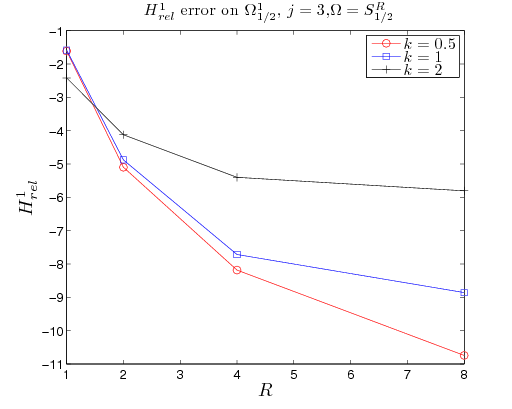}}
\caption{Numerical errors (lin-log scale) on the domain $\Omega_{1/2}^1$ by varying the semi-minor axis $R$ of $E_{1/2}^R$. }
\label{errorL2H1s1}
 \end{figure}

\section{Numerical results for a variable index of refraction} \label{section 4}
In this section we apply our approach to cases of variable indices of refraction, for which an exact solution is not generally given. We notice that most of the numerical methods available in literature can be of difficult application, since they usually need the knowledge of the exact solution in some exterior domain. We also mention that the perfectly matched layer might be of nontrivial implementation when the material is not an analytic function in the direction perpendicular to the boundary, see \cite{Jo},\cite{OZAJ}.

We stress that for a variable index of refraction it can be hard to compute the asymptotic behavior of the integrand function in \eqref{eq: intro rad con Perth Vega II} and, unlike the case of constant index of refraction, the weight $(1+|x|)^{-1}$ in the integral \eqref{eq: introd Funzionale} has to be maintained.

Thus, given the same notation as the previous section, we are solving the following problem: to find $f \in L^2(\partial \Omega)$ such that the solution $u_{\Omega}[f]$ of
\begin{equation*}
\begin{cases}
\Delta u_\Omega[f] + k^2n(x)^2 u_\Omega[f] = 0, \qquad \textmd{in } \Omega\backslash B_{\hat{r}} \\
u_\Omega[f] = g,\quad \textmd{on } \pa B_{\hat{r}} \\
\partial_\nu u_\Omega[f] = f,\quad \textmd{on } \pa \Omega,
\end{cases}
\end{equation*}
minimizes the integral
\begin{equation}\label{integrale_num}
J(u_\Omega[f])=\ints_{\Omega\backslash B_{\hat{r}} } \Big| \nabla u_\Omega[f] (x) - ikn(x)u_\Omega[f] (x) \frac{x}{|x|}\Big|^2\frac{1}{1+|x|}dx.
\end{equation}

In the following paragraphs we shall propose two indexes of refraction, namely
\begin{eqnarray*}
 n_1(x,y)&=&2+(\exp(-(x-1)^2-y^2)+\exp(-(x+1)^2-y^2)) \frac{x}{\sqrt{x^2+y^2}}, \\
n_2^a(x,y)&=&2+a\frac{x}{\sqrt{x^2+y^2}},
\end{eqnarray*}
being $a$ a non negative real number. These indexes are shown in Figs.\ref{indexrefr}a-b, with $a=0.1$. We stress that $n_1$ has an exponential decay to 2 at infinity and that $n_2^a$ it is not constant at infinity.
For both cases, the geometry of the computational domain is the annulus $\Omega_{1/2}^R=B_{R}\backslash B_{1/2}$ for which a better convergence rate is achieved in the CGM algorithm, as seen in the previous section for a constant index of refraction.

Being an exact solution not available, we test the convergence properties as follows: we set the solution for $R=16$ as the reference one, and we shall compute the error norms up to $R=8$ as done in the previous sections. We always consider $j=0,k=1$.\\
 \begin{figure}
\hspace*{-2cm}\subfigure[$n_1(x,y)$]{\includegraphics[width=8cm,height=6.5cm]{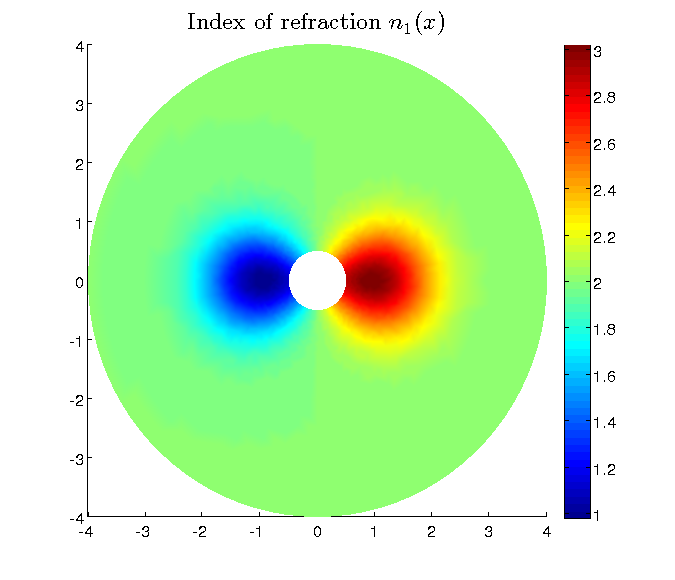}}
\subfigure[$n_2^{0.1}(x,y)$]{\includegraphics[width=8cm,height=6.5cm]{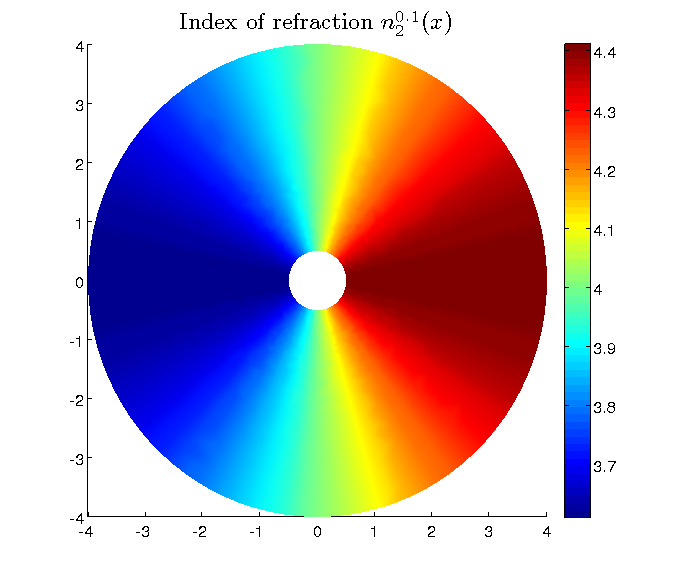}}
\caption{The indexes of refraction $n_1$ and $n_2^a$. \label{indexrefr}}
\end{figure}

We notice that the choice of the functional \eqref{integrale_num} is appropriate only when the index of refraction does not generate guided modes (trapped waves). Due to the exponential decay at infinity, this requirement is satisfied by $n_1$ (see \cite{Ag} and \cite{Sa0}). Regarding the index $n_2^a$ it is known from \cite{PV} that there is a critical value $a^{gm}$ according to which no guided modes are generated for $a\leq a^{gm}$.
It is not easy to give a quantitative estimate of $a^{gm}$. By evaluating the integral \eqref{integrale_num} for several values of $a$ (see Fig.\ref{integralen3_loglin.eps} where the integral is shown at various exterior radius $R$ in log-lin scale), it seems that for $a\gtrapprox 0.6$ the integral diverge as $R\rightarrow\infty$. This suggests to consider $a\lessapprox0.6$ for which it appears that the functional is bounded as $R\to \infty$. The numerical results for $n_2^a$ are therefore shown only for the small value $a=0.1$: if we are luckily in the case in which no guided modes are generated, our approach works well and the numerical solution converges to the exact one. More accurate estimates on $a^{gm}$ will be the object of future work.

 \begin{figure}
\begin{center}
\includegraphics[width=10cm,height=9.5cm]{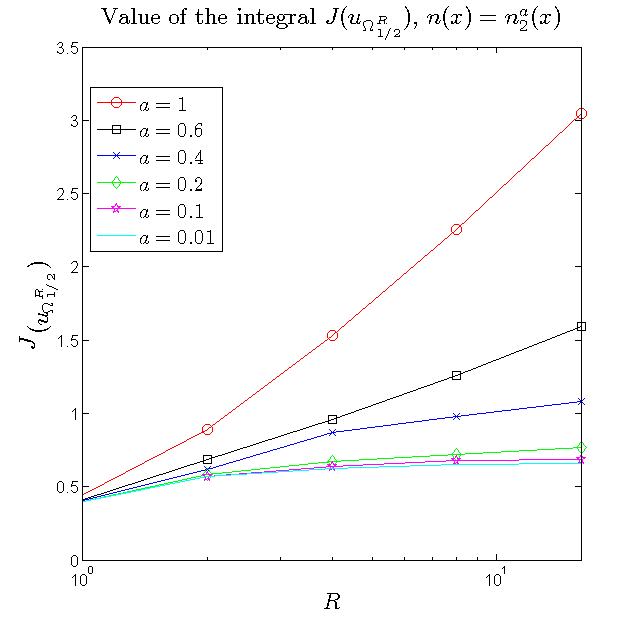}
\caption{The integral value  $J(u_{\Omega_{1/2}^R[f]})$, by varying the exterior radius $R$ for various $a$ (log-lin scale): for $a\gtrapprox0.6$ a likely numerical divergence is noticeable.}
\label{integralen3_loglin.eps}
\end{center}
\end{figure}

\paragraph{Case \textbf{$n_1(x)$}:}

 In Tables \ref{tableconvn1} we report the errors on $\Omega^1_{1/2}$ and $\Omega_{1/2}^R$ for $R=1,2,4,8$, (the $L^2_{rel},H^1_{rel}$ errors on $\Omega^1_{1/2}$
are also shown in Fig.\ref{errorn1})a) in lin-log scale.
In Figs.\ref{solutionn1}a-b the real and imaginary parts of the numerical solution $u_{\Omega_{1/2}^{R}}[f]$
are shown for $R=16$

\begin{table}[ht]
\tablefont{2.5mm}
\begin{center}
   \hspace*{-1.7cm} \begin{tabular}{|l|l|l||l|l|l|l|l||l|l|l|l|l||}
\multicolumn{3}{|c||}{} & \multicolumn{5}{|c||}{Variable index of refraction $n_1(x)$:} & \multicolumn{5}{|c||}{Variable index of refraction $n_1(x)$:} \\
\multicolumn{3}{|c||}{} & \multicolumn{5}{|c||}{Errors on $\Omega_{1/2}^1$ ($\cdot 10^{-1}$)} & \multicolumn{5}{|c||}{Errors on $\Omega_{1/2}^R$ ($\cdot 10^{-1}$)} \\ \hline
$j$ &$k$ &    $R$ &$L^{2}$ & $L^{2}_{rel}$ & $H^1$ & $H^1_{rel}$ & $\delta J_{rel}$ &$L^{2}$ & $L^{2}_{rel}$ & $H^1$ & $H^1_{rel}$ & $\delta J_{rel}$ \\ \hline
 0 &1 &     $1$ & $0.734$ & $0.582$ & $2.85$ & $1.01$& $0.088$& $ $ & $ $ & $ $ & $ $ &$ $\\
  & &     $2$ & $0.284$ & $0.225$ & $1.12$&$0.403$& $0.156$& $1.21 $ & $0.559 $ & $1.93$ & $0.499$ &$0.672$\\
& &     $4$ & $0.174$ & $0.137$ & $0.502$& $0.177 $ &$0.0511$&$0.821$& $0.242$ & $1.49$ & $0.229$ & $0.157 $ \\
  & &     $8$ & $0.0656$ & $0.0521$ & $0.203$&$0.0728$&$0.00602$& $0.411 $ & $0.0826 $ & $0.791$ & $0.0842$ &$0.00193 $\\
\hline
\end{tabular}
\caption{The $L^2,L^2_{rel},H^1_{rel},\delta J_{rel}$ errors of the numerical solution on the annulus $\Omega_{1/2}^1$ and $\Omega_{1/2}^R$ by varying the exterior radius $R$ of the annular domain $\Omega_{1/2}^R$ for the variable index of refraction $n_1(x)=2+(\exp(-(x-1)^2-y^2)+\exp(-(x+1)^2-y^2))\cos(\theta)$.\label{tableconvn1}}
\end{center}
\end{table}

 \begin{figure}
\hspace*{-2cm}\subfigure[]{\includegraphics[width=8cm,height=6.5cm]{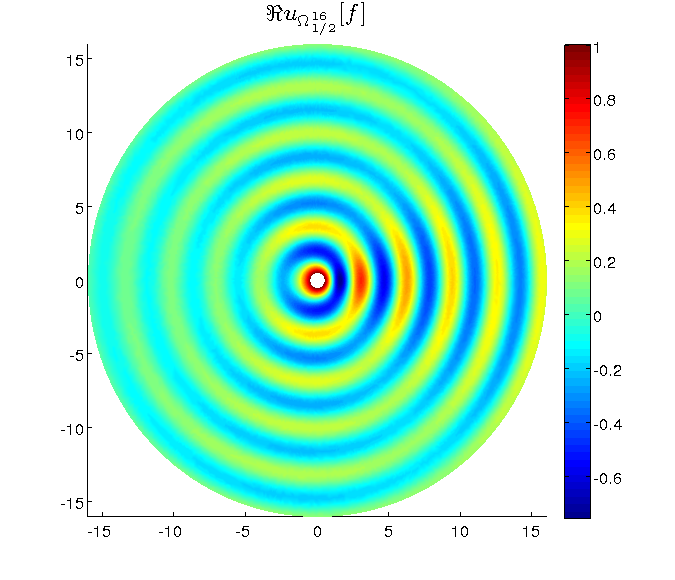}}
\subfigure[]{\includegraphics[width=8cm,height=6.5cm]{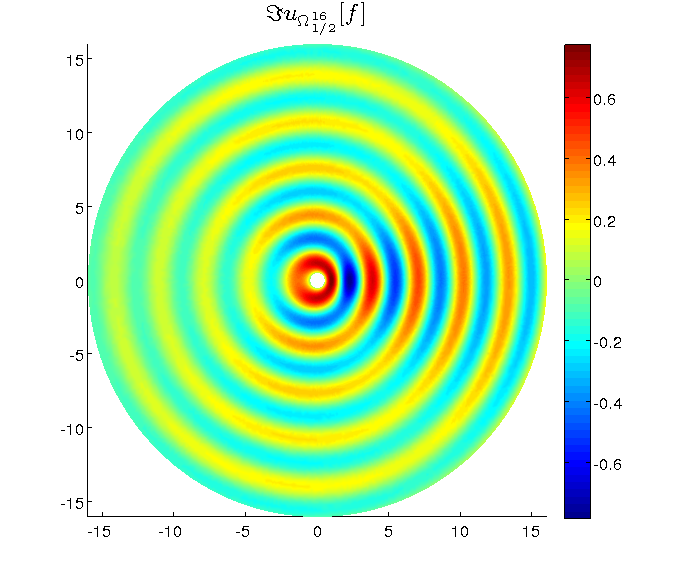}}
\caption{The real and imaginary parts of the numerical solution $u_{\Omega_{1/2}^{16}}$,index of refraction $n_1(x,y)$. }
\label{solutionn1}
 \end{figure}

 \begin{figure}
\begin{center}
\includegraphics[width=10cm,height=7.5cm]{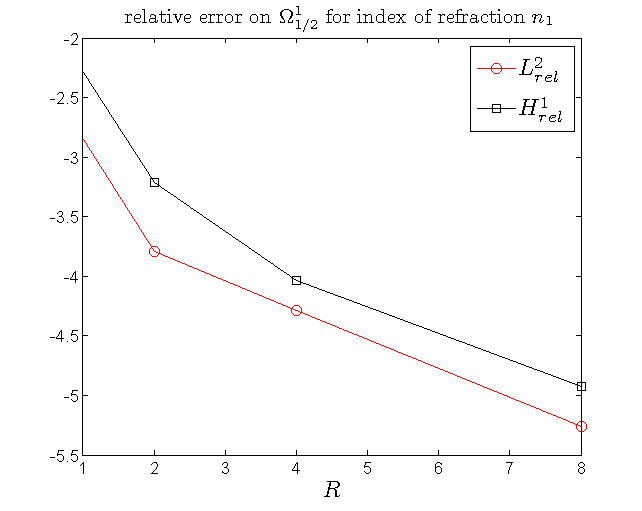}
\end{center}
\caption{Numerical errors (lin-log scale) of the solution for variable index of refraction $n_1$ on the domain $\Omega_{1/2}^1$ by varying the radious $R$ of $\Omega_{1/2}^R$. }
\label{errorn1}
 \end{figure}

\paragraph{Case \textbf{$n_2^{0.1}(x)$}:}

 In Tables \ref{tableconvn3} we report the errors on $\Omega^1_{1/2}$ and $\Omega_{1/2}^R$ for $R=1,2,4,8$  (the $L^2_{rel},H^1_{rel}$ errors on $\Omega^1_{1/2}$
are also shown in Fig.\ref{errorn3})b) in lin-log scale.
In Figs.\ref{solutionn3}a-b the real and imaginary parts of the numerical solution $u_{\Omega_{1/2}^{R}}[f]$
are shown for $R=16$.

\begin{table}[ht]
\tablefont{2.5mm}
\begin{center}
   \hspace*{-1.7cm} \begin{tabular}{|l|l|l||l|l|l|l|l||l|l|l|l|l||}
\multicolumn{3}{|c||}{} & \multicolumn{5}{|c||}{Variable index of refraction $n_2^{0.1}(x)$:} & \multicolumn{5}{|c||}{Variable index of refraction $n_2^{0.1}(x)$:} \\
\multicolumn{3}{|c||}{} & \multicolumn{5}{|c||}{Errors on $\Omega_{1/2}^1$ ($\cdot 10^{-1}$)} & \multicolumn{5}{|c||}{Errors on $\Omega_{1/2}^R$ ($\cdot 10^{-1}$)} \\ \hline
$j$ &$k$  &     $R$ &$L^{2}$ & $L^{2}_{rel}$ & $H^1$ & $H^1_{rel}$ & $\delta J_{rel}$ &$L^{2}$ & $L^{2}_{rel}$ & $H^1$ & $H^1_{rel}$ & $\delta J_{rel}$ \\ \hline
 0 &1 &    $1$ &0.522&0.408&1.65&0.582&0.826& $ $ & $ $ & $ $ & $ $ &$ $ \\
  & &     $2$ &0.221&0.179&0.715&0.249&0.442&0.531&0.237&1.03&0.238&0.0925 \\
& &    $4$ & 0.121&0.0951&0.402&0.141&0.273&0.421&0.122&0.0801&0.128&0.0833 \\
  & &    $8$ &0.0872&0.0681&0.245&0.0861&0.168&0.478&0.0951&0.911&0.101&0.0161 \\
\hline
\end{tabular}
\caption{The $L^2,L^2_{rel},H^1_{rel},\delta J_{rel}$ errors of the numerical solution on the annulus $\Omega_{1/2}^1$ and $\Omega_{1/2}^R$ by varying the exterior radius $R$ of the annular domain $\Omega_{1/2}^R$ for the variable index of refraction $n_2^{0.1}(x)=2+0.1\cos(\theta)$}
\label{tableconvn3}
\end{center}
\end{table}

 \begin{figure}
\hspace*{-2cm}\subfigure[$n_2^{0.1}(x,y)$]{\includegraphics[width=8cm,height=6.5cm]{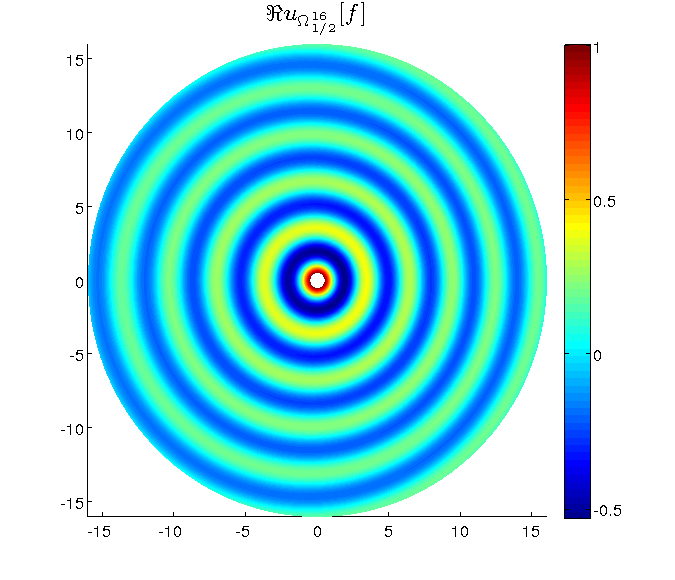}}
\subfigure[$n_2^{0.1}(x,y)$]{\includegraphics[width=8cm,height=6.5cm]{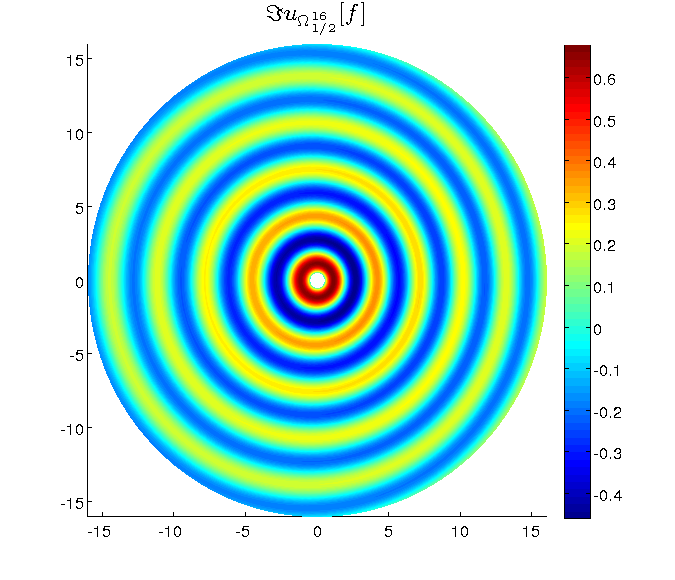}}
\caption{The real and imaginary parts of the numerical solution $u_{\Omega_{1/2}^{16}}$, index of refraction $n_2^{0.1}(x,y)$.}
\label{solutionn3}
 \end{figure}

 \begin{figure}
\begin{center}
\includegraphics[width=10cm,height=7.5cm]{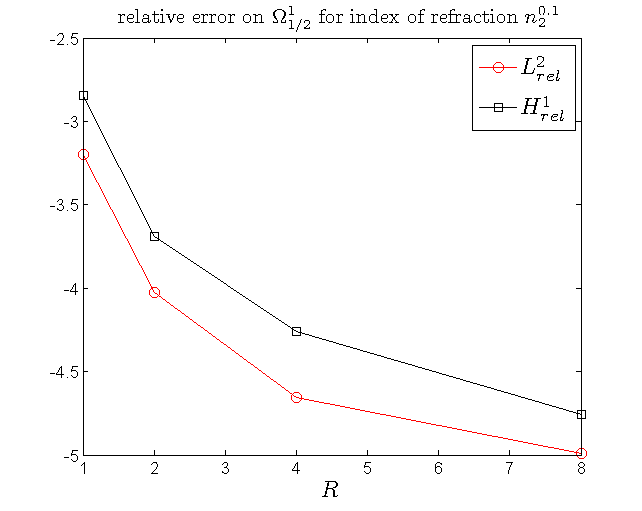}
\end{center}
\caption{Numerical errors (lin-log scale) of the solution for variable index of refraction $n_2^{0.1}$ on the domain $\Omega_{1/2}^1$ by varying the radious $R$ of $\Omega_{1/2}^R$. }
\label{errorn3}
 \end{figure}

\section*{Conclusions}
We studied a new approach to the problem of transparent boundary conditions for the Helmholtz equation. Our method is based on the minimization of an integral functional which arises from a volume integral formulation of the Sommerfeld radiation condition. When the artificial domain is a ball, we proved analytical results on the convergence of the approximate solution to the exact one under quite general assumptions on the index of refraction.
 We tested numerically our method for a constant index of refraction by choosing three shapes of the artificial boundary in $\RR^2$: a ball, an ellipse and a square. The convergence results are satisfactory; in particular we observed a good convergence rate for the $H^1_{rel}$ norm.
 We also applied our approach to other indices of refraction, one exponentially decaying at infinity to a constant, and one not constant at infinity.

Since our approach is of easy implementation and has wide application, we believe that it can be a good starting point for studying problems with nontrivial indices of refraction. However, when the index of refraction is constant, the methods available in literature are more suitable. We stress that numerical results show that our approach do not suffer the choice of the shape of the artificial domain.

Our approach applies also to more general settings. In particular, we have in mind the waveguide problem. In this case, the functional $J$ has to be modified on account of the presence of guided modes (see \cite{AC},\cite{Ci1}-\cite{CM2}); this will be the object of future work.

\end{document}